\documentclass[12pt,twopage]{paper}

\usepackage{libertine}
\usepackage{amssymb}
\usepackage{amsfonts}
\usepackage{amsmath}
\usepackage{bold-extra}
\usepackage{pifont}
\usepackage{yfonts}
\usepackage{mathrsfs}
\usepackage{url}
\usepackage{hyperref}
\usepackage{fancyhdr}

\usepackage{pmboxdraw}
 \usepackage{stmaryrd}

\newtheorem{theorem}{\textbf{\textsc{Theorem}}}[section]
\newtheorem{definition}[theorem]{\textbf{\textsc{Definition}}}
\newtheorem{lemma}[theorem]{\textbf{\textsc{Lemma}}}

\newtheorem{proposition}[theorem]{\textbf{\textsc{Proposition}}}

\newenvironment{proof}
{\noindent\mbox{\textsf{\textbf{\textsc{Proof}}}:}}
{\hfill{\scriptsize \mbox{\underline{\texttt{\em QED}\,}$\!\big|$}}\bigskip}

\title{A Reunion of \textsc{G\"odel,  Tarski, Carnap}, and \textsc{Rosser}}

\author{{\sc Saeed \  Salehi  } \\
\textrm{\large Department of Mathematics, Statistics, and Computer Science,   University of Tabriz,}\\ \textrm{\large  Bahman 29$^{\,th}$ Boulevard,  P.O.Box~51666--16471, Tabriz, IRAN.  \, E-mail:\!~\textsf{\normalsize salehipour@tabrizu.ac.ir}} }

\begin{document}

\maketitle

\pagestyle{fancy}
\lhead[ ]{ \thepage\quad{\sc Saeed Salehi} (\textgoth{2022}) }
\chead[ ]{ }
\rhead[ ]{ {\em A Reunion of G{\scriptsize \"ODEL}, T{\scriptsize ARSKI}, C{\scriptsize ARNAP}, and R{\scriptsize OSSER}}}
\lfoot[ ]{ }
\cfoot[ ]{ }
\rfoot[ ]{ }

\begin{abstract}
We unify
{\sc G\"odel}'s First Incompleteness Theorem (1931),
{\sc Tarski}'s  Undefinability 
Theorem (1933),
{\sc G\"odel-Carnap}'s Diagonal Lemma (1934),
and {\sc Rosser}'s
(strengthening of {\sc G\"odel}'s first) Incompleteness Theorem (1936),
whose proofs resemble much and use almost the same technique.

\bigskip

{\sc Tarski}'s Undefinability Theorem,
{\sc Carnap}'s Diagonal Lemma,
{\sc Rosser}'s Incompleteness Theorem,
{\sc Chaitin}'s Proof of the Incompleteness Theorem.

\end{abstract}
%


\section{Introduction}\label{sec:intro}

{\Large B}{\large ETWEEN} {\large\textsc{1930 and 1936}}, \textsc{at the} beginning of
the Golden Age of Mathematical Logic, there appeared four fundamental theorems:
\begin{enumerate}
  \item {\sc G\"odel}'s First Incompleteness Theorem in  1931; see \cite{Godel31}.\\[-4ex]
   \item ({\sc G\"odel}-){\sc Tarski}'s Truth-Undefinability  Theorem in 1933; see \cite{Murawski98,Gruber16} and \cite[f.~25 on p.~363]{Godel34}.\\[-4ex]
   \item ({\sc G\"odel}-){\sc Carnap}'s Diagonal Lemma in   1934; see \cite{Carnap34} and \cite[f.~23 on p.~363]{Godel34}.\\[-4ex]
  \item ({\sc G\"odel}-){\sc Rosser}'s Incompleteness  Theorem in  1936; see \cite{Rosser36} and \cite[p.~370]{Godel34}.
\end{enumerate}
A main part of the classic proofs of these theorems uses a common trick that constructs some suitable self-referential sentences.\footnote{\!One other basic result around that time (1938) which uses quite a similar  technique  was {\sc Kleene}'s Recursion Theorem \cite{Kleene38}; we do not consider it here, as we know of no equivalent formulation in the form of Theorem~\ref{thm:main} below.}
So, it is natural to conjecture that these theorems are equivalent, in the sense that there is a sufficiently general framework in which either all these four theorems hold together, or none holds. This paper is a continuation of \cite{Salehi22} where
some semantic forms of 1--3 were proved to be equivalent. Here, we present some syntactic formulations of 1--4 and prove their equivalence (in~\S~\ref{sec:main}); we also provide a framework in which none of 1--4 holds (it is too well known that they all  hold for sufficiently strong theories). Having this equivalence has the advantage that one can translate a proof for any of 1--4 to get an alternative proof for another. This was done earlier in \cite{Salehi20} where an alternative proof for the semantic Diagonal Lemma, and a weak syntactic formulation of it, was provided. Here, we will also answer a question left open there (in~\S~\ref{sec:app}, the Appendix) and will see one more different proof for the weak syntactic diagonal lemma (in~\S~\ref{sec:new}).

\section{A Unification of the Four Theorems}\label{sec:main}
Let us begin with a definition for the standard notion of {\sc G\"odel} coding, and fix our language.
\begin{definition}[\textsf{\textsc{G\"odel}} coding, arithmetical languages, interpretation, representing numbers by terms]
\label{def:codelang}
\noindent

\noindent
For a first-order language $\mathcal{L}$, a {\sc G\"odel} coding on $\mathcal{L}$ is  a computable injection from the syntactical expressions (finite strings) over $\mathcal{L}$ into the set of natural number $\mathbb{N}$.

\noindent
Let $\mathcal{L}^{\boldsymbol\ast}=\{\boldsymbol0,\boldsymbol1,
\boldsymbol<,
\boldsymbol+,\boldsymbol\nu\}$ be the first-order language that contains the constant symbols $\boldsymbol0,\boldsymbol1$, the binary relation symbols $\boldsymbol<$, the binary function symbol $\boldsymbol+$, and the unary function symbol $\boldsymbol\nu$.
Let the language of arithmetic be
$\{\boldsymbol0,\boldsymbol1,
\boldsymbol<,
\boldsymbol+,\boldsymbol\times\}$,
where $\boldsymbol\times$ is a binary function symbol.

\noindent
The symbols $\boldsymbol0,\boldsymbol1,
\boldsymbol<,
\boldsymbol+,\boldsymbol\times$
are interpreted as usual over $\mathbb{N}$, and the interpretation of $\boldsymbol\nu$, which is a function $\nu\colon\mathbb{N}\!\rightarrow\!\mathbb{N}$, depends on a fixed {\sc G\"odel} coding:  $\nu(n)$ is the code of  $\neg\sigma$ (the negation symbol $\neg$ appended to the left side of $\sigma$) when $n$ is the code of a sentence $\sigma$ (and $\nu(n)$ can be any arbitrary number when $n$ is not the code of any sentence).

\noindent
For a natural number $n\!\in\!\mathbb{N}$, let $\overline{n}$ denote the closed term that represents $n$, which is  $\boldsymbol0$ if $n\!=\!0$, is $\boldsymbol1$ if $n\!=\!1$, and is  $\boldsymbol
1\boldsymbol+\cdots\boldsymbol+\boldsymbol1$ ($n$ times) if $n\!>\!1$.
\hfill\ding{71}\end{definition}

 So, if $\ulcorner\!\sigma\!\urcorner$ denotes the code of a sentence $\sigma$, then $\nu(\ulcorner\!\sigma\!\urcorner)
 \!=\!
 \ulcorner\!\neg\sigma\!\urcorner$.
\begin{definition}[arithmetical theories]
\label{def:theory}
\noindent

\noindent
Let $\textit{Q}$ denote Robinson's Arithmetic over the language of arithmetic.

\noindent
Fix a G\"odel coding $\eta\mapsto\ulcorner\!\eta\!\urcorner$; for simplicity,  let us denote the closed term $\overline{\ulcorner\!\eta\!\urcorner}$
 by  $^{\scriptsize\textbf{\textSFxxxix}}\eta
^{\scriptsize\textbf{\textSFxxv}}$.
Let $\textit{Q}^{\boldsymbol-}$ be the $\mathcal{L}^{\boldsymbol\ast}$-theory axiomatized by

$({\sf A}_1)$:\; $\forall x (x\!\boldsymbol<\!\overline{n}\vee
x\!=\!\overline{n}\vee
\overline{n}\!\boldsymbol<\!x)$,
\,  for every $n\!\in\!\mathbb{N}$.

$({\sf A}_2)$:\; $\forall x (x\!\boldsymbol<\!\overline{n}\leftrightarrow
\bigvee\!\!\!\!\!\bigvee_{i<n}x\!=\!\overline{i}\,)$, \,  for every $n\!\in\!\mathbb{N}$.

$({\sf A}_3)$:\; $\boldsymbol\nu(^{\scriptsize\textbf{\textSFxxxix}}\sigma^{\scriptsize\textbf{\textSFxxv}}) =\,
 ^{\scriptsize\textbf{\textSFxxxix}}
 \neg\sigma\,\!\,^{\scriptsize\textbf{\textSFxxv}}$,
\,  for every $\mathcal{L}^{\boldsymbol\ast}$-sentence $\sigma$.
\hfill\ding{71}\end{definition}

\begin{theorem}[$\textrm{\textsc{G\"odel}}_T\!\equiv\!
\textrm{\textsc{Tarski}}_T\!\equiv\!
\textrm{\textsc{Carnap}}_T\!\equiv\!
\textrm{\textsc{Rosser}}_T$]\label{thm:main}
\noindent

\noindent
For every theory $T$ that extends $\textit{Q}^{\boldsymbol-}$, and a fixed coding,  the following are equivalent:
\begin{enumerate}

\item $\textrm{\textsc{G\"odel}}_T$: If $U$ is a consistent extension of $T$ such that for some formula $\Psi(x)$, with the only variable $x$, we have $U\vdash\sigma$ iff $U\vdash
      \Psi(^{\scriptsize\textbf{\textSFxxxix}}\sigma
^{\scriptsize\textbf{\textSFxxv}})$
for every sentence $\sigma$, then $U$ is incomplete.

\item $\textrm{\textsc{Tarski}}_T$: For every formula $\Upsilon(x)$, with the only free variable $x$, the theory $T$ is inconsistent with the set $\{\Upsilon(^{\scriptsize\textbf{\textSFxxxix}}\sigma^{\scriptsize\textbf{\textSFxxv}})\!\leftrightarrow\!\sigma\mid\sigma\textrm{ is a sentence}\}$.

\item $\textrm{\textsc{Carnap}}_T$: For every formula $\Lambda(x)$, with the only free variable $x$, there are finitely many sentences $\{A_i\}_i$ such that $T\vdash\bigvee\!\!\!\!\!\bigvee_i
      \big(\Lambda(^{\scriptsize\textbf{\textSFxxxix}}A_i\,\!
^{\scriptsize\textbf{\textSFxxv}})\!\leftrightarrow\!A_i\big)$.

\item $\textrm{\textsc{Rosser}}_T$: If $U$ is a consistent extension of $T$ such that for some formula $\Theta(x,y)$, with the shown free variables, we have for every sentence $\sigma$ that  (i) if  $U\vdash\sigma$ then $U\vdash\Theta(\overline{m},
      ^{\scriptsize\textbf{\textSFxxxix}}\!\sigma
^{\scriptsize\textbf{\textSFxxv}})$
for some number $m\!\in\!\mathbb{N}$, and (ii) if $U\nvdash\sigma$ then $U\vdash\neg\Theta(\overline{n},
      ^{\scriptsize\textbf{\textSFxxxix}}\!\sigma
^{\scriptsize\textbf{\textSFxxv}})$
      for every number $n\!\in\!\mathbb{N}$, then $U$ is incomplete.
\hfill $\diamond$
\end{enumerate}
\end{theorem}

Before proving the theorem, let us explain its content a bit.

(1) In $\textrm{\textsc{G\"odel}}_T$, the formula $\Psi(x)$ is a kind of provability predicate, in the sense that if $U$ proves $\sigma$, then $U$ can verify that it proves $\sigma$; and conversely (here $U$ is assumed to possess some soundness) if $U$ proves that $\sigma$ is $U$-provable, then $\sigma$ is $U$-provable in reality.

(2) In $\textrm{\textsc{Tarski}}_T$,  the formula $\Upsilon(x)$ is a kind of hypothetical truth predicate, and our version of the theorem is syntactic ({\sc Tarski}'s theorem is usually formulated semantically, in the form that {\em the truth predicate is not arithmetically definable}; cf.\ \cite{Salehi22}).

(3) In  $\textrm{\textsc{Carnap}}_T$, the existence of finitely many (partial) fixed points for a given formula $\Lambda(x)$ has been claimed. This is a weak version of the syntactic Diagonal Lemma (see \cite[Theorem~2.5]{Salehi20}); the strong diagonal lemma states the existence of one fixed point sentence ($A$ such that $\textit{Q}\vdash\Lambda
(^{\scriptsize\textbf{\textSFxxxix}}A\,
^{\scriptsize\textbf{\textSFxxv}})\!\leftrightarrow\!A$).
Let us note that our weak version implies the Semantic Diagonal Lemma  (see e.g.\ \cite[Theorem 2.3]{Salehi20} or \cite[Definition 2.1]{Salehi22}).

(4) In  $\textrm{\textsc{Rosser}}_T$, the formula $\Theta(x,y)$ is a kind of proof predicate: $x$ codes a $U$-proof of $y$. The assumptions (i) and (ii) indicate that $U$-proofs are bi-representable in $U$: if $U$ proves $\sigma$ and $m$ is the code of its proof, then $U$ verifies this; and if $\sigma$ is not $U$-provable, then $U$ verifies that no number can code a $U$-proof of $\sigma$.

\begin{proof}

\qquad
$\boldsymbol(\!\!(1\Longrightarrow 2)\!\!\boldsymbol)\!\!:$
If $\textrm{\textsc{Tarski}}_T$ does not hold, then let $\mathfrak{M}$ be a model of
$T+\{\Upsilon(^{\scriptsize\textbf{\textSFxxxix}}\sigma
^{\scriptsize\textbf{\textSFxxv}})
      \leftrightarrow\sigma\mid\sigma\textrm{ is a sentence}\}$,
      and put $U={\rm Th}(\mathfrak{M})$. Now, for every sentence $\sigma$ we have
     $U\vdash\sigma$ iff $\mathfrak{M}\vDash\sigma$ iff $\mathfrak{M}\vDash\Upsilon(^{\scriptsize\textbf{\textSFxxxix}}\sigma
^{\scriptsize\textbf{\textSFxxv}})$ iff  $U\vdash \Upsilon(^{\scriptsize\textbf{\textSFxxxix}}\sigma
^{\scriptsize\textbf{\textSFxxv}})$.
But $U$ is a complete extension of $T$; this contradicts $\textrm{\textsc{G\"odel}}_T$.

\qquad
$\boldsymbol(\!\!(2\Longrightarrow 3)\!\!\boldsymbol)\!\!:$
For given $\Lambda(x)$, let $\Upsilon(x)=\neg\Lambda(x)$. By the inconsistency of the theory $T$ with the set  $\{\Upsilon(^{\scriptsize\textbf{\textSFxxxix}}\sigma
^{\scriptsize\textbf{\textSFxxv}})
      \leftrightarrow\sigma\mid\sigma\textrm{ is a sentence}\}$,
      we have $T\vdash\neg\bigwedge\!\!\!\!\!\bigwedge_i
      \big(\Upsilon(^{\scriptsize\textbf{\textSFxxxix}}A_i\,\!
^{\scriptsize\textbf{\textSFxxv}})
      \leftrightarrow A_i\big)$
      for some finitely many sentences $\{A_i\}_i$.
      By the propositional tautology $\neg\bigwedge\!\!\!\!\!\bigwedge_i(\neg p_i\!\leftrightarrow\!q_i)\!\equiv\!
      \bigvee\!\!\!\!\!\bigvee_i(p_i\!\leftrightarrow\!q_i)$,
we have  $T\vdash\bigvee\!\!\!\!\!\bigvee_i
      \big(\Lambda(^{\scriptsize\textbf{\textSFxxxix}}A_i\,\!
^{\scriptsize\textbf{\textSFxxv}})
      \leftrightarrow A_i\big)$.

\qquad
$\boldsymbol(\!\!(3\Longrightarrow 4)\!\!\boldsymbol)\!\!:$
Let $U$ and $\Theta$ satisfy the assumptions, and assume, for the sake of a contradiction, that $U$ is a complete theory.
Let $\Lambda(x)$ be the formula $\forall y\big[
\Theta(y,x)\!\rightarrow\!\exists z\!<\!y\,\Theta\big(z,\boldsymbol\nu(x)\big)\big]$.
By $\textrm{\textsc{Carnap}}_T$, there are finitely many sentences $\{A_i\}_i$ such that $T\vdash\bigvee\!\!\!\!\!\bigvee_i
      \big(\Lambda(^{\scriptsize\textbf{\textSFxxxix}}A_i\,\!
^{\scriptsize\textbf{\textSFxxv}})
      \leftrightarrow A_i\big)$.
      Since complete theories have the disjunction property, then there exists one sentence $\rho$ such that (\ding{92}) $U\vdash\Lambda(^{\scriptsize\textbf{\textSFxxxix}}\rho\,
^{\scriptsize\textbf{\textSFxxv}})
      \leftrightarrow\rho$.
      Also,   by the completeness of $U$ we have either (I) $U\vdash\rho$, or (II) $U\vdash\neg\rho$.

(I)\; If $U\vdash\rho$, then by (\ding{92}) and $({\sf A}_3, \textrm{Definition}~\ref{def:theory})$ we have  $U\vdash\forall y\big[
\Theta(y,^{\scriptsize\textbf{\textSFxxxix}}\rho\,
^{\scriptsize\textbf{\textSFxxv}})\!\rightarrow\!\exists z\!<\!y\,\Theta(z,^{\scriptsize\textbf{\textSFxxxix}}\neg\rho\,
^{\scriptsize\textbf{\textSFxxv}})\big]$.
By (i) in the assumption, there is some $m\!\in\!\mathbb{N}$ such that $U\vdash\Theta(\overline{m},^{\scriptsize\textbf{\textSFxxxix}}\rho\,
^{\scriptsize\textbf{\textSFxxv}})$.
So, $U\vdash\exists z\!<\!\overline{m}\,\Theta(z,^{\scriptsize\textbf{\textSFxxxix}}\neg\rho\,
^{\scriptsize\textbf{\textSFxxv}})$.
On the other hand, by $U\nvdash\neg\rho$ and (ii) in the assumption, we have $U\vdash\neg\Theta(\overline{n},^{\scriptsize\textbf{\textSFxxxix}}\neg\rho\,
^{\scriptsize\textbf{\textSFxxv}})$
for each $n\!\in\!\mathbb{N}$. Thus, by $({\sf A}_2, \textrm{Definition}~\ref{def:theory})$ we have
$U\vdash\forall z\!<\!\overline{m}\neg\Theta(z,^{\scriptsize\textbf{\textSFxxxix}}\neg\rho\,
^{\scriptsize\textbf{\textSFxxv}})$. Whence, $U$ is inconsistent; a contradiction.

(II)\; If $U\vdash\neg\rho$, then by (\ding{92}) and $({\sf A}_3, \textrm{Definition}~\ref{def:theory})$ we have $U\vdash\exists y\big[
\Theta(y,^{\scriptsize\textbf{\textSFxxxix}}\rho\,
^{\scriptsize\textbf{\textSFxxv}})\!\wedge\!\forall z\!<\!y\,\neg\Theta(z,^{\scriptsize\textbf{\textSFxxxix}}\neg\rho\,
^{\scriptsize\textbf{\textSFxxv}})\big]$.
By (i) in the assumption, there is some $m\!\in\!\mathbb{N}$ such that $U\vdash\Theta(\overline{m},^{\scriptsize\textbf{\textSFxxxix}}\neg\rho\,
^{\scriptsize\textbf{\textSFxxv}})$.
So, $({\sf A}_1, \textrm{Definition}~\ref{def:theory})$ implies that $U\vdash\exists y\!\leqslant\!\overline{m}\,\Theta(y,^{\scriptsize\textbf{\textSFxxxix}}\rho\,
^{\scriptsize\textbf{\textSFxxv}})$.
On the other hand, by $U\nvdash\rho$ and (ii) in the assumption, we have $U\vdash\neg\Theta(\overline{n},^{\scriptsize\textbf{\textSFxxxix}}\rho\,
^{\scriptsize\textbf{\textSFxxv}})$
for each $n\!\in\!\mathbb{N}$. Thus, by $({\sf A}_2, \textrm{Definition}~\ref{def:theory})$ we have
$U\vdash\forall y\!\leqslant\!\overline{m}\neg\Theta(y,^{\scriptsize\textbf{\textSFxxxix}}\rho\,
^{\scriptsize\textbf{\textSFxxv}})$; a contradiction.

\qquad
$\boldsymbol(\!\!(4\Longrightarrow 1)\!\!\boldsymbol)\!\!:$
If $\textrm{\textsc{G\"odel}}_T$ does not hold for the theory $U$ and formula $\Psi(x)$, let $\Theta(x,y)$ be $\Psi(y)\wedge(x\!=\!x)$.
Now, by the assumption,  for every sentence $\sigma$ we have
(i) if $U\vdash\sigma$, then  $U\vdash\Psi(^{\scriptsize\textbf{\textSFxxxix}}\sigma
^{\scriptsize\textbf{\textSFxxv}})$,
thus   $U\vdash\Theta(\overline{m},^{\scriptsize\textbf{\textSFxxxix}}\sigma
^{\scriptsize\textbf{\textSFxxv}})$
for every $m\!\in\!\mathbb{N}$. Also, since $U$ is (assumed to be) complete, then by the assumption we have (ii) if $U\nvdash\sigma$, then $U\nvdash\Psi(^{\scriptsize\textbf{\textSFxxxix}}\sigma
^{\scriptsize\textbf{\textSFxxv}})$,
so  $U\vdash\neg\Psi(^{\scriptsize\textbf{\textSFxxxix}}\sigma
^{\scriptsize\textbf{\textSFxxv}})$,
thus  $U\vdash\neg\Theta(\overline{n},^{\scriptsize\textbf{\textSFxxxix}}\sigma
^{\scriptsize\textbf{\textSFxxv}})$
for every $n\!\in\!\mathbb{N}$. This contradicts $\textrm{\textsc{Rosser}}_T$.
\end{proof}

For the theorem to make sense, we should demonstrate a framework in which none of $\textrm{\textsc{G\"odel}}_T$, $\textrm{\textsc{Tarski}}_T$, $\textrm{\textsc{Carnap}}_T$, or
$\textrm{\textsc{Rosser}}_T$  holds. Let us fix a coding, which is due to {\sc Ackermann} (1937), cf.~\cite[Example~7]{CegRich99}.

\begin{definition}[\textsf{\textsc{Ackermann}} (1937) coding]
\label{def:coding}
\noindent

\noindent
 The number $0$ is not the code of anything; number $1$ is the code of the empty string; and every symbol is coded by an odd number, in particular $3$ is the code of $\neg$. Code the finite string $\langle a_1,\cdots,a_\ell\rangle$ by $\sum_{k=1}^{\ell}2^{\sum_{j=1}^{k}(a_j+1)}=
 \sum_{k=1}^{\ell}2^{(a_1+1)+\cdots+(a_k+1)}=
 2^{(a_1+1)}+2^{(a_1+1)+(a_2+1)}+\cdots+2^{(a_1+1)+\cdots+(a_\ell+1)}$.
\hfill\ding{71}\end{definition}

Let us note that this coding is computable and injective (cf.\ Definition~\ref{def:codelang}), but not surjective (for example, 4 is not the code of anything).

\begin{lemma}[the code of negation]\label{lem:code}
\noindent

\noindent
If 
$\eta\mapsto\llceil\eta\rrceil$
is {\sc Ackermann}'s coding in Definition~\ref{def:coding}, then for every string $\eta$
we have  $\llceil\neg\eta\rrceil\!=\!
16(1\!+\!\llceil\eta\rrceil)$.
\end{lemma}
\begin{proof}

\noindent
For $\eta=\langle a_1,\cdots,a_\ell\rangle$ we have
$\llceil\langle\neg,a_1,\cdots,a_\ell
\rangle\rrceil
=
2^{4}+ \sum_{k=1}^{\ell}2^{4+\sum_{j=1}^{k}(a_j+1)}=
16(1\!+\!\llceil\eta\rrceil)$.
\end{proof}

\begin{definition}[a new framework]
\label{def:model}
\noindent

\noindent
Define the function $\nu^{\boldsymbol\ast}\colon  \mathbb{N}\!\rightarrow\!\mathbb{N}$
by  $\nu^{\boldsymbol\ast}(n)\!=\!33\!+\!16n$ when $n$ is even, and
$\nu^{\boldsymbol\ast}(n)\!=\!16\!+\!16n$ when $n$ is odd.

\noindent
Let $\mathfrak{M}^{\boldsymbol\ast}$ be the structure $\langle\mathbb{N};\boldsymbol0,\boldsymbol1,
\boldsymbol<,\boldsymbol+,\boldsymbol\nu\rangle$ where $\boldsymbol\nu$ is interpreted as $\nu^{\boldsymbol\ast}$ above.

\noindent
Put $T^{\boldsymbol\ast}\!=\!{\rm Th}(\mathfrak{M}^{\boldsymbol\ast})$,
and let $\Psi^{\boldsymbol\ast}(x)
\!=\!\Upsilon^{\boldsymbol\ast}(x)
\!=\!\exists y (x\!=\!y\!\boldsymbol+\!y)$, $\Lambda^{\boldsymbol\ast}(x)
\!=\!\neg\Psi^{\boldsymbol\ast}(x)$,
and
$ \Theta^{\boldsymbol\ast}(x,y)
\!=\!\Psi^{\boldsymbol\ast}(y)\!\wedge\!(x\!=\!x)$.

\noindent
Denote {\sc Ackermann}'s coding
in Definition~\ref{def:coding} by $\eta\!\mapsto\!\llceil\eta\rrceil$.
Let $\eta\mapsto\llcorner\eta\lrcorner$ be a new coding on $\mathcal{L}^{\boldsymbol\ast}$ (see Definition~\ref{def:codelang}) defined as follows: $\llcorner\eta\lrcorner$ is $2\llceil\eta\rrceil$,
when $\eta$ is an $\mathfrak{M}^{\boldsymbol\ast}$-true $\mathcal{L}^{\boldsymbol\ast}$-sentence; and is $1\!+\!2\llceil\eta\rrceil$,
otherwise. Let $_{\scriptsize\textbf{\textSFxlix}}E\,\!_{\scriptsize\textbf{\textSFxxvii}}$ denote the closed term $\overline{\llcorner E\lrcorner}$.
\hfill\ding{71}\end{definition}

Since $\eta\!\mapsto\!
\llceil\eta\rrceil$
(in Definition~\ref{def:coding}) is injective, the new coding $\eta\mapsto\llcorner\eta\lrcorner$ (in Definition~\ref{def:model}) is injective too. For its computability, we note that   $\mathfrak{M}^{\boldsymbol\ast}=\langle\mathbb{N};\boldsymbol0,\boldsymbol1,
\boldsymbol<,\boldsymbol+,\boldsymbol\nu\rangle$
 is decidable since $\nu^{\boldsymbol\ast}$ is $\{\boldsymbol1,\boldsymbol+\}$-definable and  $\langle\mathbb{N};\boldsymbol0,\boldsymbol1,
\boldsymbol<,\boldsymbol+\rangle$ is decidable by {\sc Presburger}'s Theorem (cf.\ \cite[Theorem~3.3]{CegRich99}). Thus, $\eta\mapsto\llcorner\eta\lrcorner$ is computable as well (cf.\ Definition~\ref{def:codelang}).
We   show that $\nu^{\boldsymbol\ast}$ calculates the negation of sentences in the new coding's setting:

\begin{lemma}[$\nu^{\boldsymbol\ast}$ is the negation mapping of  sentences in the setting of  $\llcorner\cdot\lrcorner$]\label{lem:nu}
\noindent

\noindent
For every $\mathcal{L}^{\boldsymbol\ast}$-sentence $\sigma$, we have $\nu^{\boldsymbol\ast}(\llcorner\sigma\lrcorner)
=\llcorner\neg\sigma\lrcorner$.
\end{lemma}
\begin{proof}

\noindent
The sentence $\sigma$ is either $\mathfrak{M}^{\boldsymbol\ast}$-true or $\mathfrak{M}^{\boldsymbol\ast}$-false. If $\mathfrak{M}^{\boldsymbol\ast}\vDash\sigma$,
then
$\llcorner\sigma\lrcorner=
2\llceil\sigma\rrceil$
and
$\neg\sigma$ is not  $\mathfrak{M}^{\boldsymbol\ast}$-true. So, $\llcorner\sigma\lrcorner$ is even, and by Lemma~\ref{lem:code} we have $\nu^{\boldsymbol\ast}(\llcorner\sigma\lrcorner)\!=\!33\!+\!16
\llcorner\sigma\lrcorner\!=\!33\!+\!32
\llceil\sigma\rrceil\!=\!1\!+\!2\!\cdot\!
16(1\!+\!\llceil\sigma\rrceil)\!=\!1\!+\!2
\llceil\neg\sigma\rrceil\!=\!\llcorner\!\neg
\sigma\!\lrcorner$.
If $\mathfrak{M}^{\boldsymbol\ast}\nvDash\sigma$, then
$\llcorner\sigma\lrcorner=1\!+\!2\llceil\sigma\rrceil$ and
$\neg\sigma$ is  $\mathfrak{M}^{\boldsymbol\ast}$-true. So, $\llcorner\sigma\lrcorner$ is odd, thus by Lemma~\ref{lem:code} we have $\nu^{\boldsymbol\ast}(\llcorner\sigma\lrcorner)\!=\!16\!+\!16
\llcorner\sigma\lrcorner\!=\!32\!+\!32
\llceil\sigma\rrceil\!=\!2\!\cdot\!16(1\!+\!
\llceil\sigma\rrceil)\!=\!2\llceil
\neg\sigma\rrceil\!=\!\llcorner\!\neg
\sigma\!\lrcorner$.
\end{proof}

Let us note that if the expression $\eta$ is not a sentence, then Lemma~\ref{lem:nu} does not hold; indeed, in that case  $\nu^{\boldsymbol\ast}(\llcorner\eta\lrcorner)\!+\!1
\!=\!\llcorner\!\neg\eta\lrcorner$.
Now, we have all the ingredients for constructing  a case in which none of the equivalent statements $\textrm{\textsc{G\"odel}}_T$, $\textrm{\textsc{Tarski}}_T$, $\textrm{\textsc{Carnap}}_T$,
$\textrm{\textsc{Rosser}}_T$ can hold (cf.\ \cite[Remark~2.6]{Salehi22}).

\begin{theorem}[$\neg\textrm{\textsc{G\"odel}}_{T^{\boldsymbol\ast}}$, $\neg\textrm{\textsc{Tarski}}_{T^{\boldsymbol\ast}}$,  $\neg\textrm{\textsc{Carnap}}_{T^{\boldsymbol\ast}}$,
$\neg\textrm{\textsc{Rosser}}_{T^{\boldsymbol\ast}}$]\label{thm:none}
\noindent

\noindent
The complete theory $T^{\boldsymbol\ast}$ contains $\textrm{\textit{Q}}^{\boldsymbol-}$, and neither $\textrm{\textsc{G\"odel}}_{T^{\boldsymbol\ast}}$, nor $\textrm{\textsc{Tarski}}_{T^{\boldsymbol\ast}}$, nor $\textrm{\textsc{Carnap}}_{T^{\boldsymbol\ast}}$, nor
$\textrm{\textsc{Rosser}}_{T^{\boldsymbol\ast}}$ holds for $U=T^{\boldsymbol\ast}$ and, respectively,  the formulas
$\Psi^{\boldsymbol\ast}(x)$,
$\Upsilon^{\boldsymbol\ast}(x)$,
$\Lambda^{\boldsymbol\ast}(x)$,
and
$\Theta^{\boldsymbol\ast}(x,y)$, in  Definition~\ref{def:model}.
\end{theorem}
\begin{proof}

\noindent
The complete theory $T^{\boldsymbol\ast}$ is decidable by {\sc Presburger}'s theorem and the $\{\boldsymbol1,\boldsymbol+\}$-definability of $\nu^{\boldsymbol\ast}$. Thus, the mapping $\eta\mapsto\llcorner\eta\lrcorner$ is a (computable injective) coding. Trivially, the axioms $({\sf A}_1,{\sf A}_2, \textrm{Definition}~\ref{def:theory})$
hold in $\mathfrak{M}^{\boldsymbol\ast}$, and  Lemma~\ref{lem:nu} implies that  $({\sf A}_3, \textrm{Definition}~\ref{def:theory})$
holds too. Thus, $T^{\boldsymbol\ast}$ contains $\textrm{\textit{Q}}^{\boldsymbol-}$. We now show that none of
$\textrm{\textsc{G\"odel}}_{T^{\boldsymbol\ast}}$, $\textrm{\textsc{Tarski}}_{T^{\boldsymbol\ast}}$,  $\textrm{\textsc{Carnap}}_{T^{\boldsymbol\ast}}$, or
$\textrm{\textsc{Rosser}}_{T^{\boldsymbol\ast}}$ holds for
$\Psi^{\boldsymbol\ast}(x)$,
$\Upsilon^{\boldsymbol\ast}(x)$,
$\Lambda^{\boldsymbol\ast}(x)$,
and
$\Theta^{\boldsymbol\ast}(x,y)$.

$\neg\textrm{\textsc{G\"odel}}_{T^{\boldsymbol\ast}}$:
For every sentence $\sigma$, $T^{\boldsymbol\ast}\vdash\sigma$ iff $\mathfrak{M}^{\boldsymbol\ast}\vDash\sigma$ iff $\llcorner\sigma\lrcorner$ is even iff $\mathfrak{M}^{\boldsymbol\ast}\vDash
\Psi^{\boldsymbol\ast}(_{\scriptsize\textbf{\textSFxlix}}\sigma_{\scriptsize\textbf{\textSFxxvii}})$
iff $T^{\boldsymbol\ast}\vdash\Psi^{\boldsymbol\ast}(_{\scriptsize\textbf{\textSFxlix}}\sigma_{\scriptsize\textbf{\textSFxxvii}})$.

$\neg\textrm{\textsc{Tarski}}_{T^{\boldsymbol\ast}}$: For every sentence $\sigma$ we have $T^{\boldsymbol\ast}\vdash\Upsilon^{\boldsymbol\ast}(_{\scriptsize\textbf{\textSFxlix}}\sigma_{\scriptsize\textbf{\textSFxxvii}})\!\leftrightarrow\!\sigma$ by $\neg\textrm{\textsc{G\"odel}}_{T^{\boldsymbol\ast}}$.

$\neg\textrm{\textsc{Carnap}}_{T^{\boldsymbol\ast}}$:
For any finitely many sentences $\{A_i\}_i$ we have $T^{\boldsymbol\ast}\vdash
\bigwedge\!\!\!\!\!\bigwedge_i
      \big(\neg\Lambda^{\boldsymbol\ast}(_{\scriptsize\textbf{\textSFxlix}}A_i\,\!
_{\scriptsize\textbf{\textSFxxvii}})
      \leftrightarrow A_i\big)$ by $\neg\textrm{\textsc{Tarski}}_{T^{\boldsymbol\ast}}$;       so $T^{\boldsymbol\ast}\vdash\neg
\bigvee\!\!\!\!\!\bigvee_i
      \big(\Lambda^{\boldsymbol\ast}(_{\scriptsize\textbf{\textSFxlix}}A_i\,\!
_{\scriptsize\textbf{\textSFxxvii}})
      \leftrightarrow A_i\big)$,
      thus $T^{\boldsymbol\ast}\nvdash
\bigvee\!\!\!\!\!\bigvee_i
      \big(\Lambda^{\boldsymbol\ast}(_{\scriptsize\textbf{\textSFxlix}}A_i\,\!
_{\scriptsize\textbf{\textSFxxvii}})
      \leftrightarrow A_i\big)$
      by the consistency of $T^{\boldsymbol\ast}$.

$\neg\textrm{\textsc{Rosser}}_{T^{\boldsymbol\ast}}$:
We saw that (i) if $T^{\boldsymbol\ast}\vdash\sigma$, then $T^{\boldsymbol\ast}\vdash\Psi^{\boldsymbol\ast}(_{\scriptsize\textbf{\textSFxlix}}\sigma_{\scriptsize\textbf{\textSFxxvii}})$, so $T^{\boldsymbol\ast}\vdash\Theta^{\boldsymbol\ast}(\overline{m},_{\scriptsize\textbf{\textSFxlix}}\!\sigma_{\scriptsize\textbf{\textSFxxvii}})$ for every $m\!\in\!\mathbb{N}$. Also note that (ii) if $T^{\boldsymbol\ast}\nvdash\sigma$, then
$\llcorner\sigma\lrcorner$ is odd, so  $T^{\boldsymbol\ast}\vdash\neg\Psi^{\boldsymbol\ast}(_{\scriptsize\textbf{\textSFxlix}}\sigma_{\scriptsize\textbf{\textSFxxvii}})$,
 thus $T^{\boldsymbol\ast}\vdash\neg\Theta^{\boldsymbol\ast}(\overline{n},_{\scriptsize\textbf{\textSFxlix}}\!\sigma_{\scriptsize\textbf{\textSFxxvii}})$ for every $n\!\in\!\mathbb{N}$.
\end{proof}

\section{\textsf{\textsc{Tarski}}'s Theorem and the Weak Syntactic Diagonal Lemma, \'a la \textsf{\textsc{Chaitin}}}\label{sec:new}
{\sc Chaitin}'s proof for the first incompleteness theorem appeared in   \cite{Chaitin71}. There are several versions of it now; one was presented  in \cite[Theorem~3.3]{SalSer18}. The proof was adapted for  {\sc Rosser}'s theorem   in \cite[Theorem~3.9]{SalSer18}. Here, we prove {\sc Tarski}'s undefinability theorem and the (weak syntactic) diagonal lemma of {\sc Carnap} by the same method.

\begin{proposition}[\textsf{\textsc{G\"odel-Tarski}'s Truth-Undefinability Theorem}]\label{prop:tarski}
\noindent

\noindent
For every formula $\Upsilon(x)$,  the theory $\textit{Q}$ is inconsistent with  the set $\{\Upsilon(^{\scriptsize\textbf{\textSFxxxix}}\sigma^{\scriptsize\textbf{\textSFxxv}})\!\leftrightarrow\!\sigma\mid\sigma\textrm{ is a sentence}\}$.
\end{proposition}
\begin{proof}

\noindent
Assume not; fix a model $\mathfrak{M}$ of
$\textit{Q}+\{\Upsilon(^{\scriptsize\textbf{\textSFxxxix}}\sigma^{\scriptsize\textbf{\textSFxxv}})\!\leftrightarrow\!\sigma\mid\sigma\textrm{ is a sentence}\}$.
Suppose that $\boldsymbol\varphi_0^\Upsilon,
\boldsymbol\varphi_1^\Upsilon,
\boldsymbol\varphi_2^\Upsilon,\cdots$
is an effective enumeration of all the $\Upsilon$-computable (i.e., computable with oracle $\Upsilon$) unary functions. Define the $\Upsilon$-{\sc Kolmogorov}-{\sc Chaitin} complexity of   $n\!\in\!\mathbb{N}$ to be $\mathscr{K}^\Upsilon\!(n)\!=\!\min\{i\mid
\boldsymbol\varphi_i^\Upsilon\!(0)\!=\!n\}$,
the minimum index of the $\Upsilon$-computable function that outputs $n$ on input $0$. By {\sc Kleene}'s Recursion Theorem \cite{Kleene38} there exists some $\textswab{c}\!\in\!\mathbb{N}$ such that $\boldsymbol\varphi_{\textswab{c}}^\Upsilon\!(x)\!=\!\min z\!\colon\!\Upsilon(\ulcorner\!\langle\!\!\langle
\mathscr{K}^\Upsilon\!(\overline{z})\!>\!
\overline{\textswab{c}}
\rangle\!\!\rangle\!\urcorner)$,
where
$\langle\!\!\langle
\mathscr{K}^\Upsilon(x)\!>\!y
\rangle\!\!\rangle$
is the arithmetical formula which says that ``the $\Upsilon$-{\sc Kolmogorov}-{\sc Chaitin} complexity of $x$ is greater than $y$''.\footnote{Notice that this very proof implies that the unary function $z\!\mapsto\!\mathscr{K}^\Upsilon\!(z)$ is not $\Upsilon$-computable; though, for every constant $c\!\in\!\mathbb{N}$, the function $x\!\mapsto\!\min z\!\colon\!\Upsilon(\ulcorner\!\langle\!\!\langle
\mathscr{K}^\Upsilon\!(\overline{z})\!>\!
\overline{c}
\rangle\!\!\rangle\!\urcorner)$
is clearly $\Upsilon$-computable.}
By the Pigeonhole Principle (a version of which is provable in $\textit{Q}$, see \cite[Lemma~3.8]{SalSer18}) there exists some element $u\!\leqslant\!\textswab{c}\!+\!1$ in $\mathfrak{M}$ such that $\mathfrak{M}\vDash\langle\!\!\langle
\mathscr{K}^\Upsilon\!(\overline{u})\!>\!\overline{\textswab{c}}
\rangle\!\!\rangle$;
note that    $\{\boldsymbol\varphi_0^\Upsilon(0),
\boldsymbol\varphi_1^\Upsilon(0),
\cdots,
\boldsymbol\varphi_{\textswab{c}}^\Upsilon(0)\}$ has  at most $\textswab{c}\!+\!1$ members, and  $\{0,1,\cdots,\textswab{c}\!+\!1\}$ has $\textswab{c}\!+\!2$ members. For the least $\textswab{u}\!\in\!\mathfrak{M}$ with $\mathfrak{M}\vDash\langle\!\!\langle
\mathscr{K}^\Upsilon\!(\overline{\textswab{u}})\!>\!\overline{\textswab{c}}
\rangle\!\!\rangle$
we have
$\mathfrak{M}\vDash
\Upsilon(\ulcorner\!\langle\!\!\langle
\mathscr{K}^\Upsilon\!(\overline{\textswab{u}})\!>\!
\overline{\textswab{c}}
\rangle\!\!\rangle\!\urcorner)$
and $\mathfrak{M}\vDash
\forall z\!<\!\overline{\textswab{u}}\,\neg
\Upsilon(\ulcorner\!\langle\!\!\langle
\mathscr{K}^\Upsilon\!(\overline{z})\!>\!
\overline{\textswab{c}}
\rangle\!\!\rangle\!\urcorner)$;
  so $\mathfrak{M}\vDash\langle\!\!\langle
\boldsymbol\varphi_{\textswab{c}}^\Upsilon\!(\overline{0})\!=\!\overline{\textswab{u}}
\rangle\!\!\rangle$,
thus
$\mathfrak{M}\vDash\langle\!\!\langle
\mathscr{K}^\Upsilon\!(\overline{\textswab{u}})\!\leqslant\!\overline{\textswab{c}}
\rangle\!\!\rangle$, a contradiction.
\end{proof}

The part $\boldsymbol(2\Longrightarrow 3\boldsymbol)$ in the proof of Theorem~\ref{thm:main} immediately yields the weak syntactic diagonal lemma from the above proof. However, this {\sc Chaitin} style argument can   prove the weak  lemma directly.

\begin{proposition}[weak syntactic \textsf{\textsc{G\"odel-Carnap}}'s Diagonal Lemma]\label{prop:wdl}
\noindent

\noindent
For every formula $\Lambda(x)$  there are finitely many sentences $\{A_i\}_i$ such that   $\textit{Q}\vdash\bigvee\!\!\!\!\!\bigvee_i
      \big(\Lambda(^{\scriptsize\textbf{\textSFxxxix}}A_i\,\!
^{\scriptsize\textbf{\textSFxxv}})\!\leftrightarrow\!A_i\big)$.
\end{proposition}
\begin{proof}

\noindent
Put $\Upsilon(x)\!=\!\neg\Lambda(x)$; with the notation of the proof of Proposition~\ref{prop:tarski},  let $A_i=\langle\!\!\langle
\mathscr{K}^\Upsilon\!(\overline{i})\!>\!\overline{\textswab{c}}
\rangle\!\!\rangle$
 for $i\!\leqslant\!\textswab{c}\!+\!1$. If $\textit{Q}\nvdash\bigvee\!\!\!\!\!\bigvee_{i\leqslant\textswab{c}+1}
      \big(\Lambda(^{\scriptsize\textbf{\textSFxxxix}}A_i\,\!
^{\scriptsize\textbf{\textSFxxv}})\!\leftrightarrow\!A_i\big)$,
then $\textit{Q}+\bigwedge\!\!\!\!\!\bigwedge_{i\leqslant\textswab{c}+1}
      \big(\Upsilon(^{\scriptsize\textbf{\textSFxxxix}}A_i\,\!
^{\scriptsize\textbf{\textSFxxv}})\!\leftrightarrow\!A_i\big)$ is consistent, and so  has a model, say, $\mathfrak{M}$.   Now, continue the proof of Proposition~\ref{prop:tarski} (after the footnote {\footnotesize{\bf 2}}) for reaching to    a contradiction.
\end{proof}

\section{Appendix: Constructivity of an Alternative Proof of the Diagonal Lemma}\label{sec:app}
For $m,n\!\in\!\mathbb{N}$, let $\boldsymbol\delta({m},{n})$ say that the formula with    code $m$ has exactly one free variable and defines the number $n$; that is, if $\varphi(x)$ is the formula with code $m$ that has exactly one free variable $x$,   then the statement  $\forall x [\varphi(x)\leftrightarrow x\!=\!\overline{n}]$
holds. It was proved in \cite[Theorem~2.3]{Salehi20} that for every   formula $\Lambda(x)$ there are some  $m,n\!\in\!\mathbb{N}$ such that $\mathbb{N}\vDash
\Lambda\big(\ulcorner\boldsymbol\delta(\overline{m},
\overline{n})\urcorner\big)\!\leftrightarrow\!\boldsymbol\delta(\overline{m},\overline{n}).$
The proof  was not constructive, it only showed the mere existence of some $m,n\!\in\!\mathbb{N}$ with the above property; it did not determine  which $m,n$. 
By  a suggestion of a referee of \cite{Salehi20}, for a restricted class of $\Lambda$ formulas, such $m,n$ can be found constructively; but it was left open if there exists a constructive way of finding  such $m,n$ for every formula $\Lambda(x)$. Here, we show that there is such a way, but with a very different method. Actually, the following proof is more similar to the classical one (rather than to the proof of Theorem~2.3 in \cite{Salehi20}).

\begin{theorem}[strong syntactic \textsf{\textsc{G\"odel-Carnap}'s Diagonal Lemma}]\label{prop:sdl}
\noindent

\noindent
For every given formula $\Lambda(x)$ one can effectively find    $m,n\!\in\!\mathbb{N}$ such
 that $\textit{Q}\vdash
\boldsymbol\delta(\overline{m},\overline{n})
\!\leftrightarrow\!
\Lambda\big(\,\!^{\scriptsize\textbf{\textSFxxxix}}
\boldsymbol\delta(\overline{m},
\overline{n})\,\!^{\scriptsize\textbf{\textSFxxv}}\big).$
\end{theorem}
\begin{proof}

\noindent
There is a formula $\boldsymbol\sigma(x,y)$ that strongly represents the diagonal function in $\textit{Q}$. That is to say that for every formula $\alpha(y)$ we have $\textit{Q}\vdash\forall x[\boldsymbol\sigma(\overline{a},x)
\!\leftrightarrow\!
x\!=\!\,\!^{\scriptsize\textbf{\textSFxxxix}}
\alpha(\overline{a})\,\!^{\scriptsize\textbf{\textSFxxv}}]$,
where  $a=\ulcorner\!\alpha\!\urcorner$; the sentence $\alpha(\overline{a})$ is called the diagonal of $\alpha$.
Let $\zeta(x,y)=[\boldsymbol\sigma(y,x)\!\rightarrow\!\Lambda(x)]
\!\leftrightarrow\!(x\!=\!y)$
and  $\tau(y)=\forall x\,[\zeta(x,y)\!\leftrightarrow\!x\!=\!y]$;
put $n=\ulcorner\!\tau\!\urcorner$. Also, let $\kappa(x)=\zeta(x,\overline{n})$ and put $m=\ulcorner\!\kappa\!\urcorner$. We show that $\boldsymbol\delta(\overline{m},\overline{n})
\!\leftrightarrow\!
\Lambda\big(\,\!^{\scriptsize\textbf{\textSFxxxix}}
\boldsymbol\delta(\overline{m},
\overline{n})\,\!^{\scriptsize\textbf{\textSFxxv}}\big)$
is provable in $\textit{Q}$. Note that $\boldsymbol\delta(\overline{m},
\overline{n})=\forall x[\kappa(x)\!\leftrightarrow\!
x\!=\!\overline{n}]=\forall x[\zeta(x,\overline{n})\!\leftrightarrow\!
x\!=\!\overline{n}]=\tau(\overline{n})\!=$ 
 the diagonal of $\tau$.
Thus,

\noindent
\begin{tabular}{rclcl}
  $\textit{Q}\vdash\boldsymbol\delta(\overline{m},
\overline{n})$
& $\longleftrightarrow$ & $\forall x[\zeta(x,\overline{n})\!\leftrightarrow\!
x\!=\!\overline{n}]$  & $\quad$  & by what was shown above, \\
& $\longleftrightarrow$ &  $\forall x\big[\big([\boldsymbol\sigma(\overline{n},x)
\!\rightarrow\!\Lambda(x)]\!\leftrightarrow\!
x\!=\!\overline{n}\big)\!\leftrightarrow\!
x\!=\!\overline{n}\big]$ & $\quad$  &   by the definition of $\zeta$, \\
& $\longleftrightarrow$ & $\forall x\big([\boldsymbol\sigma(\overline{n},x)
\!\rightarrow\!\Lambda(x)]\!\leftrightarrow\!
[x\!=\!\overline{n}
\!\leftrightarrow\!x\!=\!\overline{n}]\big)$  & $\quad$  & by the associativity of $\leftrightarrow$, \\
& $\longleftrightarrow$ & $\forall x[\boldsymbol\sigma(\overline{n},x)\!\rightarrow\!\Lambda(x)]$  & $\quad$  &  by logic, \\
& $\longleftrightarrow$ & $\forall x[x\!=\!\,\!^{\scriptsize\textbf{\textSFxxxix}}
\tau(\overline{n})\,\!^{\scriptsize\textbf{\textSFxxv}}
\!\rightarrow\!\Lambda(x)]$  & $\quad$  & 
by $n\!=\!\ulcorner\!\tau\!\urcorner$  and  the property of $\boldsymbol\sigma$, \\
& $\longleftrightarrow$ & $\forall x[x\!=\!\,\!^{\scriptsize\textbf{\textSFxxxix}}
\boldsymbol\delta(\overline{m},
\overline{n})\,\!^{\scriptsize\textbf{\textSFxxv}}
\!\rightarrow\!\Lambda(x)]$  & $\quad$  & by $\boldsymbol\delta(\overline{m},
\overline{n})=\tau(\overline{n})$ shown above, \\
& $\longleftrightarrow$ & $\Lambda\big(\,\!^{\scriptsize\textbf{\textSFxxxix}}
\boldsymbol\delta(\overline{m},
\overline{n})\,\!^{\scriptsize\textbf{\textSFxxv}}\big)$  & $\quad$  & by logic. \\
\end{tabular}

\vspace{-3.45ex}
\end{proof}

\end{document}